\documentclass[a4paper]{article}
\usepackage[T1]{fontenc}
\usepackage{amsmath,amssymb,amsthm,mathrsfs}
\usepackage{cite}
\usepackage{mathtools}
\usepackage{enumitem}

\usepackage[pdftex,bookmarksopen=true,bookmarks=true,unicode,setpagesize]{hyperref}
\hypersetup{colorlinks=true,linkcolor=black,citecolor=black}

\usepackage{pgf,tikz}
\usetikzlibrary{arrows,patterns}

\let\equation=\gather
\let\endequation=\endgather

\allowdisplaybreaks[3]
\numberwithin{equation}{section}

\makeatletter
\renewcommand*{\@fnsymbol}[1]{\ensuremath{\ifcase#1\or 1\or 2\or
   3\or 4\else\@ctrerr\fi}}
\makeatother

\newtheorem{theorem}{Theorem}[section]
\newtheorem{corollary}[theorem]{Corollary}
\newtheorem{lemma}[theorem]{Lemma}
\newtheorem{proposition}[theorem]{Proposition}
\theoremstyle{definition}

\theoremstyle{remark}
\newtheorem{remark}[theorem]{Remark}

\newcounter{assum}
\newenvironment{assum}[1][]{\ifx\newenvironment#1\newenvironment\refstepcounter{assum}\fi\equation\tag{\ensuremath{\mathrm{A}\theassum#1}}}{\endequation}

\newcommand{\R}{\mathbb{R}}
\newcommand{\X}{{\mathbb{R}^d}}
\newcommand{\N}{\mathbb{N}}
\newcommand{\F}{\mathcal{F}}
\newcommand{\B}{\mathcal{B}}

\newcommand{\E}{\mathbb{E}}
\newcommand{\xinf}{\mathcal{X}_{\infty}}
\newcommand{\xt}{\mathcal{X}_T}
\newcommand{\x}{\mathcal{X}}
\newcommand{\xtb}[2]{\mathcal{X}_{T}^{#1,#2}}
\newcommand{\xtbt}[3]{\mathcal{X}_{#3}^{#1,#2}}

\newcommand{\om}{\omega}
\newcommand{\la}{\lambda}
\newcommand{\intred}{\int\limits_{\X}}

\newcommand{\ra}{\rightarrow}
\newcommand{\inte}[2]{\int\limits_{#1}^{#2}}
\newcommand{\ka}{\varkappa}
\newcommand{\kap}{\varkappa^+}
\newcommand{\kam}{\varkappa^-}
\renewcommand{\k}{\kappa}

\title{Global stability in a nonlocal reaction-diffusion~equation}

\author{Dmitri Finkelshtein\thanks{Department of Mathematics,
Swansea University, Singleton Park, Swansea SA2 8PP, U.K. ({\tt d.l.finkelshtein@swansea.ac.uk}).} 
\and Yuri Kondratiev\thanks{Fakult\"{a}t
f\"{u}r Mathematik, Universit\"{a}t Bielefeld, Postfach 110 131, 33501 Bielefeld,
Germany ({\tt kondrat@math.uni-bielefeld.de}).} 
\and Stanislav Molchanov\thanks{Department of Mathematics and Statistics, University of North Carolina Charlotte, NC 28223, USA ({\tt smolchan@uncc.edu}); National Research University ``Higher School of Economics'', Russia.}
\and Pasha Tkachov\thanks{Fakult\"{a}t
f\"{u}r Mathematik, Universit\"{a}t Bielefeld, Postfach 110 131, 33501 Bielefeld,
Germany ({\tt ptkachov@math.uni-bielefeld.de}).}}

\date{\today}

\begin{document}

\maketitle

\begin{abstract}
We study stability of stationary solutions for a class of non-local semilinear parabolic equations. To this end, we prove the Feynman--Kac formula for a  L\'{e}vy processes with time-dependent potentials and arbitrary initial condition. We propose sufficient conditions for asymptotic stability of the zero solution, and use them to the study of the spatial logistic equation arising in population ecology. For this equation, we find conditions which imply that its positive stationary solution is asymptotically stable. We consider also the case when the initial condition is given by a random field.


\end{abstract}

\section{Introduction}
\subsection{Overview of results}
The aim of this paper is to study stability of stationary solutions to a class of non-local semilinear parabolic equations applying the Feynman--Kac formula. Namely, we wish to investigate bounded solutions to the following equation
\begin{equation}\label{eq:intro_semilin}
	\dfrac{\partial }{\partial t} u(x,t) = (L_J u)(x,t)+V(u(x,t)) u(x,t),\quad x\in\X,\ t>0,
\end{equation}
where $L_J = J*u-\|J\|_{L^1} u$, cf.~\eqref{eq:jumpgen}, is a generator of a pure jump Markov process, $V:C_b(\X)\to C_b(\X)$ is a bounded locally Lipschitz mapping, and the initial condition $u(\cdot,0)\in C_b(\X)$ belongs to a neighborhood of $u\equiv 0$.
Note that the Feynman--Kac formula for diffusion processes with time-dependent potentials  is known (see \cite[Theorem~5.7.6]{KS1991}). However, the corresponding result for general L\'{e}vy processes seems to be proved only recently in \cite{Per2012}, where compactly supported smooth initial conditions where assumed. 
We relax assumptions on the initial conditions, considering bounded continuous functions $C_b(\X)$ and prove the Feynman--Kac formula for the generator $L_J$ (see Propositions~\ref{prop:F_K_formula_general} and~\ref{prop:F_K_formula}). 
We propose also sufficient conditions for the asymptotic stability of the zero solution to \eqref{eq:intro_semilin} uniformly in space (see Theorem \ref{thm:stability_abstract}), and apply this to a particular equation   
\begin{equation}\label{eq:intro_BDLP}
\dfrac{\partial }{\partial t} u(x,t) = \kap(L_{a^+} u)(x,t)+\kam\bigl(\theta-(a^-*u)(x,t)\bigr) u(x,t), 
\end{equation}
where $\ka^\pm,m>0$, $\theta: = \frac{\kap-m}{\kam}>0$, and $a^\pm$ are probability kernels (see \cite{Mol1972, Mol1972a, PS2005, FKKozK2014, BP1997, FM2004}).
The equation \eqref{eq:intro_BDLP} may be considered as a non-local version of the classical logistic equation (see \eqref{eq:logisticODE} below). 
There are two constant solution to \eqref{eq:intro_BDLP}, $u\equiv0$ and $u\equiv\theta$.
Different properties and the long-time behavior of solutions to \eqref{eq:intro_BDLP}, were considered in \cite{FKT2015}.

We are interested to find sufficient conditions  which ensure that a solution to \eqref{eq:intro_BDLP} converges to the constant non-zero solution $u\equiv\theta$ uniformly in space. 
Applying Theorem \ref{thm:stability_abstract}, we prove (see Theorem \ref{thm:BDLP_theta_exp_stable_by_FK}) that a bounded initial condition, which is separated from zero, tends to $\theta>0$ exponentially fast if only $J_\theta(x) = \kap a^+(x) - \theta \kam a^-(x)\geq 0$, for all $x\in\X$. 
The condition on $J_\theta$ may be relaxed under more restrictive assumptions on the initial condition. 
Namely, introducing a parameter in the initial condition and considering the analytical decomposition of the corresponding solution with respect to the parameter, one can show that if $\|J_\theta\|_{L^1}<\kap$ and if the initial condition lies in a ball centered at $\theta$, then the solution tends to $\theta$ exponentially fast (see Theorem \ref{thm:BDLP_theta_exp_stable_anal_dec}). 
An example of a parameter constructed by a stationary random field provides an enhanced asymptotic for the convergence (see Theorem \ref{thm:random_field}).

\subsection{Basic notations}
Let $\B(\X)$ be the Borel $\sigma$-algebra on the $d$-dimensional Euclidean space $\X$,  $d\geq 1$. Let $C_b(\X)$ and $B_b(\X)$ denote the spaces of all bounded continuous, respectively, bounded Borel measurable functions on~$\X$. The functional spaces become Banach ones being equipped with the norm
\begin{equation*}
	\lVert v\rVert_\infty:=\sup_{x\in\X}\lvert v(x)\rvert.
\end{equation*}

For any $J\in L^1(\X):=L^1(\X,dx)$ and $v\in B_b(\X)$, one can define the classical convolution 
\[
	(J*v)(x):=\int_\X J(x-y)v(y)\,dy.
\]

Let $J\in L^1(\X)$ be non-negative. Consider the following bounded operator (in any of the Banach spaces above)
\begin{equation}\label{eq:jumpgen}
	(L_J v)(x):=\int_\X J(x-y) \bigl(v(y)-v(x)\bigr) \,dy=(J*v)(x)-\mu v(x),
\end{equation}
where $\mu:=\int_\X J(y)\,dy>0$.

Let $X_t$ be a jump-process with the state space $(\X,\B(\X))$ and the natural filtration $\F_t=\sigma(X_s\mid s\leq t)$ whose generator is $L_J$ (for details see \cite{EK1986}). 
It is well known that $X_t$ is a Markov process and, for all $s,t\geq0$, $f\in B_b(\X)$, the following equation holds,
\begin{align}\label{eq:Markov_property}
	\E \bigl[f(X_{t+s})\vert X_t\bigr] &= \E \bigl[f(X_{t+s})|\mathcal{F}_t\bigr] \nonumber \\
													 	&= \inte{\X}{} f(y) p(X_t-y,s) dy = (p*f)(X_t).
\end{align}
where $p(x,t)$ is the transition density of $X_t$. Namely, $p(x,0)=\delta(x)$ and $p(x,t)$ satisfies the following equation
\begin{equation*}
	\dfrac{\partial p}{\partial t}(x,t) =(L_{J}p)(x,t), \qquad x\in\X,\ t>0.
\end{equation*}

For an interval $I\subset\R_+:=[0,\infty)$, consider the Banach spaces $C_b(I\to E)$ of continuous $E$-valued functions on $I$, where $E$ is a space above, with the norm
\begin{equation*}
\lVert u\rVert_I:=\sup_{t\in I} \lVert u(\cdot,t)\rVert_\infty.
\end{equation*}
For the simplicity of notations, we set, for any $T_2>T_1\geq0$, $T>0$,
\begin{equation*}
	\x_{T_1,T_2}:= C_b\bigl([T_1,T_2],C_b(\X)\bigr), \quad \xt:=\x_{0,T},\quad \xinf:=C_b\bigl(\R_+,C_b(\X)\bigr),
\end{equation*}
with the corresponding norms $\|\cdot\|_{T_1,T_2}$, $\|\cdot\|_T$, $\|\cdot\|$.

\section{The Feynman--Kac Formula and Stability}
Let $u=u(x,t)$ describe the local density of a system at the point $x\in\X$, $d\geq 1$, at the moment of time $t\in \R_+$. 
Prove now a version of the Feynman--Kac formula for the time-dependent potential and operator $L_J$, cf.~e.g. \cite[Theorem~2.5]{DvC2000}, \cite[Theorem~5.7.6]{KS1991}. Consider a perturbed equation
\begin{equation}\label{eq:pert}
	\begin{cases}
		\dfrac{\partial }{\partial t} u(x,t) = (L_J u)(x,t)+W(x,t) u(x,t),\quad t\in[0,T],\\
		u(x,0)=u_{0}(x)\in C_b(\X),
	\end{cases}
\end{equation}
where $W\in\xt$. Then, clearly, \eqref{eq:pert} has a unique solution in $\xt$. The following theorem states that the solution will satisfy the Feynman--Kac formula.

\begin{proposition}\label{prop:F_K_formula_general}
        Let $u$ solves \eqref{eq:pert} for $t\in[0,T]$. Then 
\begin{equation}\label{eq:FKtime}
	u(x,t)=\E^x u_0(X_t) \exp\biggl(\int_0^t W(X_{t-s},s)\,ds\biggr), \quad x\in\X, t\in[0,T].
\end{equation}
\end{proposition}

\begin{proof}

For $f\in\xt$, we denote
	\begin{equation}\label{eq:T_def}
		(Qf)(x,t) = \inte{0}{t} \inte{\X}{} p(x-y,t-s) W(y,s) f(y,s) dy ds.
	\end{equation}
By Duhamel's formula,
\begin{align}
	u(x,t) & =  (p*u_0)(x,t) + \bigl(Qu\bigr)(x,t) = (p*u_0)(x,t) + \bigl(Q(p*u_0 + Qu)\bigr)(x,t) \nonumber \\
				 & =  (p*u_0)(x,t) + \bigl(Q(p*u_0)\bigr)(x,t) + Q^2\bigl((p*u_0) + Qu\bigr)(x,t) \nonumber \\
			   & =  \dots \textit{( by induction )} \dots \nonumber \\
				 & =  \sum\limits_{j=0}^n \bigl( Q^j(p*u_0) \bigr)(x,t) + \bigl( Q^{n+1}u \bigr)(x,t). \label{eq:Duhamel}
\end{align}
By \eqref{eq:Markov_property} and \eqref{eq:Duhamel},
\begin{align*}
	(p*u_0)(X_{t-s},s) & = \E\bigl[ u_0(X_t) \vert \F_{t-s} \bigr], \nonumber \\
	\bigl(Q(p*u_0)\bigr)(X_{t-s},s) & = \inte{0}{s} \inte{\X}{} p(X_{t-s}-y,s-\tau) W(y,\tau) (p*u_0)(y,\tau) dy d\tau \nonumber \\
							 & = \inte{0}{s} \E \bigl[W(X_{t-\tau},\tau) (p*u_0)(X_{t-\tau},\tau) \vert \F_{t-s} \bigr] d\tau \nonumber \\
							 & = \inte{0}{s} \E \Bigl[W(X_{t-\tau},\tau)\,
							 \E\bigl[ u_0(X_t) \vert \F_{t-\tau}\bigr] \Bigm\vert \F_{t-s} \Bigr] d\tau \nonumber \\
							 & = \E \biggl[u_0(X_t) \inte{0}{s} W\bigl(X_{t-\tau},\tau\bigl) d\tau \biggm\vert \F_{t-s} \biggr],  
\end{align*}
where the last equality holds by the tower rule and Fubini's theorem. 
One can continue then
\begin{align*}
	&\quad(Q^2(p*u_0))(X_{t-s},s) \\&= \inte{0}{s} \E \bigl[W(X_{t-\tau},\tau) (Q(p*u_0))(X_{t-\tau},\tau) \big\vert \F_{t-s} \bigr]  d\tau  \\
	& = \inte{0}{s} \E \biggl[W(X_{t-\tau},\tau)\, \E\Bigl[ u_0(X_t) \inte{0}{\tau}  W(X_{t-\sigma}, \sigma)d\sigma \Bigm\vert \F_{t-\tau} \Bigr] \biggm\vert \F_{t-s} \biggr]  d\tau  \\
	& =  \E \biggl[u_0(X_t) \inte{0}{s} \inte{0}{\tau} W(X_{t-\tau},\tau) W(X_{t-\sigma},\sigma) d\sigma d\tau \biggm\vert \F_{t-s} \biggr]  \\
	& =  \frac{1}{2} \E \biggl[u_0(X_t) \inte{0}{s} W(X_{s-\tau},\tau) d\tau \biggm\vert \F_{t-s} \biggr].
\end{align*}	
In the same manner, we can prove, by the induction, the following equality
\begin{equation}\label{eq:Q_n}
	(Q^n (p*u_0))(X_{t-s},s) = \frac{1}{n!} \E \biggl[u_0(X_t)  \Bigl( \inte{0}{s} W\bigl( X_{s-\tau},\tau \bigl) d\tau \Bigr)^n  \biggm\vert \F_{t-s} \biggr].
\end{equation}
 
By \eqref{eq:Duhamel} and \eqref{eq:Q_n} (with $s=t$),
\begin{align}
	u(x,t) &= \sum_{j=0}^n \E^x \bigl[\bigl( Q^j(p*u_0) \bigr) \bigl( X_0,t\bigr) \bigr] + \bigl( Q^{n+1}u \bigr)(x,t) \nonumber \\
				 &= \E^x \biggl[u_0(X_t) \sum_{j=1}^n \frac{1}{n!} \Bigl( \inte{0}{s} W\bigl( X_{s-\tau},\tau \bigl) d\tau \Bigr)^n \biggr] + \bigl( Q^{n+1}u \bigr)(x,t).
	\label{eq:some_representation_of_u}
\end{align}
We write $\tilde{Q}$ for the operator defined by \eqref{eq:T_def} with $W$ substituted by $|W|$. 
It follows easily that the equation similar to \eqref{eq:Q_n} holds for $\tilde{Q}$. In particular, for $u_0\equiv 1$ and $s=t$, 
\begin{equation*}
	(\tilde{Q}^n 1)(x,t) = \E^x \bigl[(\tilde{Q}^n 1)(X_0,t) \bigr]  = \frac{1}{n!} \E^x \biggl[\Bigl( \inte{0}{t} \bigl\lvert W(X_{t-\tau},\tau)\bigr\rvert d\tau \Bigr)^n\biggr].
\end{equation*}
Hence
\begin{equation*}
	\lVert  Q^{n} u \rVert_T \leq \lVert \tilde{Q}^{n} 1 \rVert_T \, \lVert u\rVert_T \leq \frac{T^n}{n!}\lVert W\rVert^n_T\, \lVert u\rVert_T.
\end{equation*}
As a result, for $n\ra\infty$, \eqref{eq:some_representation_of_u} yields \eqref{eq:FKtime}, that completes the proof.
\end{proof}

Consider now a general semi-linear evolution equation with the generator~$L_J$:
\begin{equation}\label{eq:semilinear}
\begin{cases}
\dfrac{\partial }{\partial t} u(x,t) = (L_J u)(x,t)+V(u(x,t)) u(x,t),\quad t>0\\
u(x,0)=u_{0}(x)\in C_b(\X),
\end{cases}
\end{equation}
where $V:C_b(\X)\to C_b(\X)$ is a bounded locally Lipschitz mapping, i.e. 
\begin{equation}\label{eq:loclip}
	\lVert V(u)\rVert_\infty\leq M_c \lVert u\rVert_\infty, \qquad \lVert V(u)-V(v)\rVert_\infty \leq M_c \lVert u-v\rVert_\infty, 
\end{equation}
provided that $\lVert u\rVert_\infty\leq c$, $\lVert v\rVert_\infty\leq c$.
Then, evidently, 
$\lVert V(u)u-V(v)v\rVert_\infty\leq (1+c)M_c\lVert u-v\rVert_\infty$, and hence, by e.g. \cite[Theorem~1.4]{Paz1983}, there exists a $T_{\max{}}\leq \infty$, such that the initial-value problem \eqref{eq:semilinear} has a unique mild solution $u$ on $[0,T_{\max{}})$, i.e. $u$ that solves the integral equation
\begin{equation*}
	u(x,t) = e^{tL_J}u_0(x)+\int_0^t e^{(t-s)L_J}\bigl(V(u(x,s)) u(x,s)\bigr)\,ds.
\end{equation*}
Moreover, $T_{\max{}}<\infty$ implies that $\lim\limits_{t\uparrow T_{\max{}}}\lVert u(\cdot,t)\rVert_\infty = \infty$. Note also that since $L_J$ is a bounded operator, then the mild solution will be classical one, i.e. $u\in \xt$, for any $T<T_{\max{}}$, and $u(x,t)$ is differentiable in $t$ w.r.t. the norm in $C_b(\X)$.

By Proposition \ref{prop:F_K_formula_general}, the following Feynman--Kac-type expression holds for the solution to \eqref{eq:semilinear}.
\begin{proposition}\label{prop:F_K_formula}
Let \eqref{eq:loclip} hold and $u$ be the unique classical solution to \eqref{eq:semilinear} on $[0,T]$, $T<T_{\max{}}$. Then
\begin{equation}\label{eq:FK}
	u(x,t)=\E^x \biggl[u_0(X_t)\exp\Bigl(\int\limits_0^t V(u(X_{t-s},s))\,ds\Bigr)\biggr],\quad x\in\X,\ t\in [0,T].
\end{equation}
\end{proposition}

Denote 
\[
\xtb{k}{l}=\{f\in\xinf|k\leq f(x,t)\leq l,\quad x\in\X,\ t\in[0,T]\}, \ k,l\in\R.
\] 
The following theorem provides sufficient conditions for the stability of the stationary solution $u\equiv 0$ to \eqref{eq:semilinear},
\begin{theorem}\label{thm:stability_abstract}
	Let there exist $p:\R^2\ra\R_+$ such that, for any $k\leq0$, $l\geq0$, 
	\begin{equation}
	\begin{aligned}
		\bigl(Vf\bigr)(x,t) \leq& - p(k,l),&&x\in\X,\ t\in[0,T],\ f\in\xtb{k}{l},\\
		p(k,l) \leq&\ p(\la k, \la l),&&\la\in[0,1]. 
	\end{aligned}\label{eq:p_monotone} 
	\end{equation}
	Suppose that $u_0\in E$ is such that, for some $c\leq 0$ and $d\geq 0$,
	\[
	c\leq u_0(x)\leq d,\quad x\in\X.
	\]
	Then, for any $T>0$, there exists a unique $u\in \xt$, which satisfies the Feynman-Kac formula \eqref{eq:FK}. Moreover, $u\in\xtbt{c}{d}{\infty}$, $\Vert u(\cdot,t)\Vert_E$ does not increase in time, and if $p(c,d)>0$, then $\Vert u(\cdot,t)\Vert_E$ converges to zero exponentially fast, namely,
	\begin{equation}\label{eq:ljapunov_exp}
		\limsup_{t\rightarrow\infty}\dfrac{\ln\Vert u_t \Vert}{t}\leq -p(0,0).
	\end{equation}
\end{theorem}
\begin{proof}
	Let us introduce the following operator: we set, for a $w\in\xinf$,
	\begin{equation*}
		[\Psi w_{t}](x)=\E^x \biggl[u_0\bigl(\eta(t)\bigr) \exp\Bigl(\int\limits_0^t[Vw_{t-s}]\bigl(\eta(s)\bigr)ds\Bigr)\biggr],\quad x\in\X,\ t\in I.
	\end{equation*}
	Then, for any $w\in\xtb{c}{d}$,
	\begin{equation}\label{eq:Psi_est}
		ce^{-tp(c,d)}\leq [\Psi w_{t}](x) \leq de^{-tp(c,d)},\quad x\in\X, t\in[0,T].
	\end{equation}
	Since $p$ is non-negative, one gets $\Psi(\xtb{c}{d})\subset \xtb{c}{d}$.
	Since $\vert e^{-x}-e^{-y} \vert \leq \vert x-y \vert$ for all $x,y\geq0$, then,	for all $v, w\in\xtb{c}{d}$, $t\in[0,T]$, the following estimate holds
	\begin{align*}
		\big\vert [\Psi w_t](x)-[\Psi v_t](x) \big\vert \leq dTM\|v-w\|_\infty,
	\end{align*}
	where $M=M_{\max\{-c,d\}}$ is defined by \eqref{eq:loclip}. 
	Hence $\Psi$ is a contraction map on $\xtb{c}{d}$ for $T=\dfrac{1}{2dM}$. Therefore, there exists a fixed point $u\in\xtb{c}{d}$. By \eqref{eq:Psi_est}, the function $u$ satisfies the following estimate
	\begin{equation*}
		ce^{-tp(c,d)}\leq u(x,t) \leq de^{-tp(c,d)},\quad x\in\X,\ t\in[0,T].
	\end{equation*}
	Hence, $c_1\le u(x,T)\le d_1$, $x\in\X$, where $c_1=ce^{-Tp(c,d)}$, $d_1=de^{-Tp(c,d)}$. 
	We can repeat the proof on $[T,2T]$ to extend $u$ to $\xtbt{c}{d}{2T}$, so that the following estimate holds
	\begin{equation*}
		c_1e^{-(t-T)p(c_1,d_1)}\leq u(x,t) \leq d_1e^{-(t-T)p(c_1,d_1)},\quad x\in\X,\ t\in[T,2T].
	\end{equation*}
	By induction, $u$ can be extended to $\xtbt{c}{d}{nT}$, and for any $n\in \N$, $x\in\X$, $t\in[nT,(n+1)T]$,
	\begin{equation}\label{eq:psi_est_n}
		c_n e^{-(t-nT)p(c_n,d_n)}\leq u(x,t) \leq d_n e^{-(t-nT)p(c_n,d_n)},
	\end{equation}
	where $c_0=c$, $d_0=d$, $c_n=c_{n-1} e^{-Tp(c_{n-1},d_{n-1})}$ and $d_n=d_{n-1} e^{-Tp(c_{n-1},d_{n-1})}$. 
	Hence, there exists a unique $u\in\xtbt{c}{d}{\infty}$, such that \eqref{eq:FK} and \eqref{eq:psi_est_n} hold. 
	Since $p$ is non-negative, $\{c_n\}$ is increasing and $\{d_n\}$ is decreasing. Moreover,
	\[
		c_n = \lambda c_{n-1},\quad  d_n = \lambda d_{n-1}, \quad \lambda = e^{-Tp(c_{n-1},d_{n-1})}\in [0,1]
	\]
	together with \eqref{eq:p_monotone} yield that $p(c_k,d_k)\leq p(c_n,d_n)$, for $k\leq n$.
	Therefore, $\Vert u(\cdot,t)\Vert_E$ does not increase in time.
	Let us prove by induction the following inequalities
	\begin{align}
		c_k e^{-T(n-k)p(c_{k},d_{k})}&\leq c_n,\qquad 0\leq k\leq n, \label{eq:cn_est} \\
		d_k e^{-T(n-k)p(c_{k},d_{k})}&\geq d_n , \qquad 0\leq k\leq n. \label{eq:dn_est}
	\end{align}
	The case $n=1$ is obvious. Let \eqref{eq:cn_est} and \eqref{eq:dn_est} hold for  $0\leq k\leq N$. We prove them for $0\leq k\leq N+1$. Since $c_N \leq 0$, we have
	\begin{equation*}
		c_{N+1}=c_Ne^{-Tp(c_N,d_N)}\geq c_N e^{-Tp(c_k,d_k)}\geq c_k e^{-\kam T(N+1-k)p(c_k,d_k)}.
	\end{equation*}
	Hence \eqref{eq:cn_est} is proved. Similarly, the following estimate yields \eqref{eq:dn_est}
	\begin{equation*}
		d_{N+1}=d_Ne^{-Tp(c_N,d_N)}\leq d_N e^{-Tp(c_k,d_k)}\geq d_k e^{-T(N+1-k)p(c_k,d_k)},
	\end{equation*}
	where $0\leq k\leq N$.
	By \eqref{eq:cn_est} and \eqref{eq:dn_est} with $k=0$, both $\{c_n\}$ and $\{d_n\}$ converge to zero exponentially fast if only $p(c,d)>0$. 
	Therefore, for $t\in [nT,(n+1)T]$, 
	$$\dfrac{\ln \Vert u_t \Vert}{t}\leq \dfrac{\ln\max \{d_n,-c_n\}}{Tn},$$
	and, by \eqref{eq:cn_est}, \eqref{eq:dn_est}, we have, for $k\geq0$,
	\begin{equation*}
		\limsup_{t\rightarrow\infty}\dfrac{\ln\Vert u_t \Vert}{t}\leq -p(c_k,d_k).
	\end{equation*}
	As a result, \eqref{eq:ljapunov_exp} holds, because $p(c_k,d_k)$ is increasing in $k$. This proves the theorem.
\end{proof}

\section{Spatial logistic equation}
We will consider the following equation for a bounded function $u(x,t)$, which describes the (approximate) value of the local density of a system of particles distributed in $\R^d$ according to the so-called spatial logistic model. More detailed explanation and historical remarks can be found in \cite[Subsection~6.1]{FKT2015}. Namely, let $u_t(x):=u(x,t)$, $x\in\X$, $t\in I$, solves the equation
\begin{equation}\label{eq:basic0}
		\dfrac{\partial u_{t}}{\partial t}=\kap[L_{a^+}u_{t}](x)+(\kap-m)u_t(x)-\kam u_t(x)(a^-*u_t)(x).
\end{equation}
In particular, $u(x,0)=u_0(x)$, $x\in\X$. Here $\kap L_{a^+}$ is a generator of the underlying random walk, cf. \eqref{eq:jumpgen}:
\begin{equation*}
[L_{a^+}h](x)=(a^+*h)(x)-h(x),\quad x\in\X,
\end{equation*}
which spends exponentially distributed random time $\tau$ in each particular position $x$, $P\{\tau>\rho\}=e^{-\kap \rho}$, and it makes a jump $x\rightarrow x+X,$ thereafter, where the random variable $X$ has the distribution density $a^+(x)$. The constant $\beta=\kap-m>0$ is the difference between the biological rate $\beta$ of the birth of a new particle and the mortality rate $m$. The last term in \eqref{eq:basic0} describes the competition between particles, the potential $\kam a^-(x-y)$ presents the interaction between two particles located at the points $x, y\in\X$. Equation \eqref{eq:basic0} is similar to the well-known logistic ordinary differential equation:
\begin{equation}\label{eq:logisticODE}
	\dfrac{\partial w}{\partial t}(t)=\beta w(t)-\kam w(t)^2,
\end{equation}
whose partial solution is the constant $\theta:=\dfrac{\beta}{\kam}$. All other positive solutions tend exponentially fast to $\theta$. The equation \eqref{eq:basic0} has the same solution $u(x,t)\equiv\theta$, $x\in\X$, $t\ge0$ (we suppose $\int\limits_{\R}a^{\pm}(y)dy=1$). This important particular solution is the exponentially stable attractor for \eqref{eq:basic0}.
We will study the neighborhood of the attractor, using variations of $u_0$.
Let us denote, for any $h\in E$,
\[
[Fh](x)=(\kap -m)h(x)-\kam h(x)(a^-*h)(x), \quad x\in\R^d.
\]
Then \eqref{eq:basic0} has the following form
\begin{equation}\label{eq:basic}
\begin{cases}
\dfrac{\partial u_{t}}{\partial t}(x)=\kap [L_{a^{+}}u_{t}](x)+[Fu_{t}](x)
& x\in\R^d,\ t\in I,\\
u(x,0)=u_0(x) & x\in\R^d.
\end{cases}
\end{equation}

The analysis of the non-linear parabolic equation \eqref{eq:basic0} will be based on integral equations. The first of them is given through the standard Duhamel's formula.
\begin{lemma}
Function $u$ solves \eqref{eq:basic} iff it satisfies the following equation
\begin{equation}\label{eq:basic_weak}
u_t=e^{\kap tL_{a^+}}u_0+\int\limits_0^te^{-(t-s)\kap L_{a^+}}[Fu_s]ds,\
t\in I.
\end{equation}
\end{lemma}
This equation has the Volterra form and can be used for the existence-uniqueness theory (see \cite{FKT2015}).
\begin{theorem}\label{thm:ex_un}
Let $u_0\in E$ be non-negative. Then, for each $T>0$, there exists a unique non-negative solution to \eqref{eq:basic} in $\xt$.
\end{theorem}
Now we will estimate solution to \eqref{eq:basic} from below. Let $u_0$ be a constant function $$u_0\equiv q_0\in(0,\theta).$$
Then the corresponding solution to \eqref{eq:basic} is the function
\begin{equation*}
	q_t=\dfrac{\theta}{1+e^{-\beta t}(\frac{\theta}{q_0}-1)}.
\end{equation*}
By Theorem \ref{thm:ex_un}, this solution is unique.
Let us fix $\kappa\in[0,\theta]$. We make the following assumption
\begin{assum}\label{ass:aplus_geq_aminus}
	 J_\k(x) := \kap a^+(x) - \k\kam a^-(x) \geq 0,\qquad x\in\X.
\end{assum}
\begin{theorem}
Let \eqref{ass:aplus_geq_aminus} hold with $\k=q_0\in(0,\theta)$. Suppose that 
\[
	u_0(x)\ge q_0,\quad x\in\X,
\]
where $u_0\in E$.  Then the corresponding to $u_0$ solution $u_t$ to \eqref{eq:basic} satisfies the following inequality
\begin{equation*}
	u_t(x)\ge q_t,\quad x\in\X,t>0.
\end{equation*}
\end{theorem}
\begin{proof}
Let us fix $T>0$. Define  $v_t=e^{Kt}(u_t-q_t)$, $t\in[0,T]$, where $K$ will be defined
later. The function $v_t$ satisfies the following linear equation
\begin{align*}
	\dfrac{\partial v_t}{\partial t}(x)&=[G_{t}v_t](x) && x\in\X,\ t\in [0,T], \\
	v_0(x)&=u_0(x)-q_0 			&& x\in\X, 
\end{align*}
where, for all $w\in E$,
\begin{equation*}
	[G_{t}w]:=\kap (a^+*w)-\kam q_t(a^-*w)-\kam w(a^-*u_t)-mw+Kw.
\end{equation*}
By Theorem \ref{thm:ex_un}, there exists $M>0$, such that
\[
u_t(x)\le M,\quad x\in\X,\ t\in[0,T].
\]
Define $K:=\kam M+m$. Since $q_t \leq q_0$ for $t \geq 0$, we have, by \eqref{ass:aplus_geq_aminus} with $\k=q_0$, that $[G_tw](x)$ is non-negative for all $t\in [0,T]$ and for all non-negative $w\in E$.
Therefore,
\begin{equation*}
	v_t(x)=\exp\biggl(\int\limits_0^t G_s ds\biggr)v_0(x)\ge0,\quad x\in\X,\ t>0,
\end{equation*}
since $v_0$ is non-negative. Hence, $u_t(x)\ge q_t$, $x\in\X$, $t\in[0,T]$. Since $T$ is arbitrary, the same holds for any $t>0$.
\end{proof}
\begin{remark}[cf. {\cite[Proposition 3.4]{FKT2015}}]
	In a similar way, it can be shown that if \eqref{ass:aplus_geq_aminus} holds with $\k=\theta$, and $u_0\in E$ is such
that $0\le u_0(x)\le\theta$, $x\in\X$, then the corresponding
solution satisfies the following inequality
$$0\le u_t(x)\le\theta,\quad x\in\X,t>0.$$
\end{remark}

The following theorem shows, based on the Feynman--Kac formula, that $u_t$ satisfies another integral equation.
\begin{theorem}\label{thm:FK_for_BDLP}
	Let \eqref{ass:aplus_geq_aminus} holds with $\k=\theta$. 
	Suppose that $u\in\xt$, $T\in(0,\infty]$, is the solution to \eqref{eq:basic0} with an initial condition $u_0\in E$. Then $u$ satisfies the following formula, for all $x\in\X$, $t\in [0,T]$,
	\begin{equation*}
		u(x,t)= \theta + \E^x \biggl[\bigl(u_0(X_t)-\theta\bigr) \exp\Bigl(-\kam\int\limits_0^t\bigl( a^-*u_{t-s}\bigr)(X_s)ds\Bigr)\biggr].
	\end{equation*}
	where $X_0=x$.
\end{theorem}
\begin{proof}
	Let us denote $g_t:=u_t-\theta$. If $u_t$ solves \eqref{eq:basic0}, then $g_t$ satisfies the following equation
\begin{equation}\label{eq:u-theta}
	\begin{cases}
		\dfrac{\partial g_{t}}{\partial t}(x)=[L_{J_\theta}g_{t}](x)-\kam g_ta^-*g_t-\beta g_t, & x\in\X,\ t\in I,\\
		g(x,0)=g_0(x)=u_0(x)-\theta, & x\in\X,
	\end{cases}
\end{equation}
where $L_{J_\theta}$ is defined by \eqref{eq:jumpgen}, for $J=J_\theta=\kap a^+ {-} \kam \theta a^- $.
We set
\begin{equation*}
	[Vh](x)=-\kam(a^-*h)(x)-\beta,\quad x\in\X, h\in E.
\end{equation*}
For such $V$ and the generator $L_J=L_{J_\theta}$ of the jump-process $X_t$, we apply Proposition \eqref{prop:F_K_formula} to the solution of \eqref{eq:u-theta}
	\begin{equation}\label{eq:FK_for_BDLP_g}
		g(x,t)=\E^x \biggl[g_0(X_t)\exp\Bigl(\int\limits_0^t[Vg_{t-s}](X_s)ds\Bigr)\biggr],\quad x\in\X,\ t\in [0,T].
	\end{equation}
Substituting $g_t=u_t-\theta$ into the previous representation completes the proof.
\end{proof}
The following theorem shows the asymptotic stability of the positive stationary solution.
\begin{theorem}\label{thm:BDLP_theta_exp_stable_by_FK}
	Let \eqref{ass:aplus_geq_aminus} holds with $\k=\theta$. Suppose that $u_0\in E$ is an initial condition to \eqref{eq:basic0}, such that
	\[ 
		c_1\leq u_0(x)\leq c_2,\quad x\in\X,
	\]
	where $0\leq c_1\leq \theta$ and $c_2 \geq \theta$.
	Then there exists a unique solution $u\in\xinf$ to \eqref{eq:basic0}. 
	Moreover, $u\in\xtbt{c_1}{c_2}{\infty}$, $\Vert u_t-\theta\Vert_E$ does not increase in time, and if $c_1>0$, then $\Vert u_t-\theta\Vert_E$ converges to zero exponentially fast, namely 
	\begin{equation*}
		\limsup_{t\rightarrow\infty}\dfrac{\ln\Vert u_t -\theta\Vert}{t}\leq -\beta.
	\end{equation*}
\end{theorem}
\begin{proof}
	We consider $g_t = u_t - \theta$. By Theorem~\ref{thm:ex_un} and \ref{thm:FK_for_BDLP}, there exists a unique solution $g\in\xinf$ to \eqref{eq:u-theta}, and this solution satisfies \eqref{eq:FK_for_BDLP_g}. 
	The rest of the proof follows from Theorem~\ref{thm:stability_abstract} with $p(c,d)=p(c)= \kam(\theta+c)$.
\end{proof}

\section{Stability on the initial condition}
We will be interested in initial conditions of the following
form
\begin{equation}\label{eq:incond}
	u_0(x,\lambda)=\theta e^{\lambda \xi(x)},
\end{equation}
where $\xi:\X\rightarrow \R$.
Since the operator $L_{a^+}$ is linear and bounded on $E$, and $F$ is analytic on $E$, then the solution $u$ to \eqref{eq:basic} depends analytically on the initial condition $u_0$ (see e.g. \cite[Theorem 3.4.4, Corollary 3.4.5, 3.4.6]{Hen1981}).
Hence, by \eqref{eq:incond}, the $E$-valued function $\lambda\mapsto u(\cdot,t,\lambda)$ is analytic on $\R$ for each $t\geq0$. Therefore, for all $\la\in\R$, it is given by the following series
\begin{equation}\label{eq:expan_in_la}
	u(\cdot ,t,\la)=\sum\limits_{n\ge0}\dfrac{\la^n}{n!}k_{n,t}(\cdot),
\end{equation}
where
\begin{equation*}
	k_{n,t}(\cdot) := \dfrac{\partial^n u}{\partial \la^n}(\cdot,t,0)\in E,\quad n\ge0.
\end{equation*}
We substitute \eqref{eq:expan_in_la} in \eqref{eq:basic_weak}.
\begin{align*}
	\sum\limits_{n\ge0}\dfrac{\la^n}{n!}k_{n,t}&=e^{\kap tL_{a^+}}\sum\limits_{n\ge0}\dfrac{\theta\la^n\xi^n}{n!} \\
	& \qquad +\int\limits_0^te^{-(t-s)\kap L_{a^+}}[F\sum\limits_{n\ge0}\dfrac{\la^n}{n!}k_{n,s}]ds.  
\end{align*}
Hence, the $n$-th Taylor coefficient satisfies the following equation
\begin{multline*}
	k_{n,t}=\theta [e^{\kap tL_{a^+}}\xi^n] +\int\limits_0^te^{-(t-s)\kap L_{a^+}}\Bigl((\kap -m)k_{n,s} \\
	  -\kam \sum\limits_{l=0}^n \binom{n}{l}k_{l,s}(a^-*k_{n-l,s})\Bigr)ds,\qquad n\ge0. 
\end{multline*}
Therefore,
\begin{align}
	\dfrac{\partial k_{n,t}}{\partial t}(x)&=\kap [L_{a^{+}}k_{n,t}](x)+(\kap-m)k_{n,t}(x) \nonumber \\
	 &\quad-\kam\sum\limits_{l=0}^n \binom{n}{l}k_{l,t}(x)(a^-*k_{n-l,t})(x), && x\in\X,\ t\in I, \label{eq:kn}
\end{align}
where $k_{n,0}(x)=\theta\xi^n(x)$.
\begin{theorem}\label{thm:BDLP_theta_exp_stable_anal_dec}
	Let $\xi \in E$ and $\gamma = \kap - \|J_\theta\|_{L^1} > 0$. 
	Then the following estimate holds
	\begin{equation*}
		\|u_t(\cdot,\la)-\theta\|_E \le \theta e^{-\gamma t}\left(\dfrac{\gamma}{2\beta}-\sqrt{\frac{\gamma^2}{4\beta^2}-\bigl(e^{|\la|\, \|\xi\|_E} -1 \bigr)\frac{\gamma}{\beta}}\right),
	\end{equation*}
	if only $|\la| < \dfrac{1}{\|\xi\|_E} \ln\Bigl(\dfrac{\gamma}{4\beta}+1\Bigr)$.
\end{theorem}
\begin{proof}
We will estimate $k_n$, $n\ge 0$.
By \eqref{eq:kn}, $k_0\equiv\theta$. The function $k_1$ satisfies the following equation
\begin{equation*}
	\dfrac{\partial k_{1,t}}{\partial t}(x) = \bigl(J_{\theta}*k_{1,t}\bigr)(x)-\kap k_{1,t}(x),\quad x\in\X,\ t\in I,  
\end{equation*}
where $k_{1,0}(x)=\theta\xi(x)$.
Since $\gamma = \kap - \|J\|_{L^1}$, we have
\begin{equation}\label{eq:k1_est}
	k_{1,t}(x) \leq e^{-\gamma t} \|k_{1,0}\|_E = \theta e^{-\gamma t} \|\xi\|_E,\quad x\in\X,\ t\geq0.
\end{equation}
Suppose that
\begin{equation*}
	\|k_{l,t}\|_E\le C_l\theta e^{-\gamma t}\|\xi\|_E^l,\quad t\in I,\ 1\le l\le n-1,
\end{equation*}
where $C_l$ is a positive constant. (Note that, by \eqref{eq:k1_est}, $C_1=1$).
Estimate $k_n$.  By~the mild form of \eqref{eq:kn}, the following inequality holds
\begin{align*}
	\|k_{n,t}\|_E&\le e^{-\gamma t}\|k_{n,0}\|_E+\kam\int\limits_0^te^{-\gamma (t-s)}\sum\limits_{l=1}^{n-1} \binom{n}{l}\|k_{l,s}\|_E\|k_{n-l,s}\|_Eds\\
		&\le e^{-\gamma t} \|k_{n,0}\|_E + \kam\int\limits_0^t e^{-\gamma (t+s)} \sum\limits_{l=1}^{n-1} \binom{n}{l}C_lC_{n-l} \theta^2\|\xi\|_E^nds \nonumber \\
		&\le \left( 1+ \frac{\beta}{\gamma}\sum\limits_{l=1}^{n-1} \binom{n}{l}C_lC_{n-l} \right)\theta\|\xi\|_E^ne^{-\gamma t}. \nonumber
\end{align*}
Therefore,
by induction,
\begin{equation}\label{eq:kn_estim}
	\|k_{n,t}\|_E\le\theta C_{n}\|\xi\|_E^ne^{-\gamma t},\quad n\ge1,
\end{equation}
where
\begin{equation}\label{eq:Cn}
	C_n=\left( 1+ \frac{\beta}{\gamma} \sum\limits_{l=1}^{n-1} \binom{n}{l}C_lC_{n-l} \right),\quad C_1=1.
\end{equation}
Put $C_0=0$. Consider the following generating function:
\begin{equation*}
	H(x):=\sum\limits_{n\ge0}\dfrac{C_n}{n!}x^n.
\end{equation*}
By \eqref{eq:Cn}, $H$ satisfies the following equation:
\begin{equation*}
	H(x) = e^x-1 + \frac{\beta}{\gamma} H^2(x).
\end{equation*}
Since $H(0)=C_0=0$ and the function $z\ra\sqrt{1-z}$ is analytic for $|z|<1$, one has 
\begin{equation}\label{eq:genf}
	H(x)=\frac{\gamma}{2\beta}-\sqrt{\frac{\gamma^2}{4\beta^2}-(e^x-1)\frac{\gamma}{\beta}},\qquad x < \ln\bigl( \frac{\gamma}{4\beta}+1\bigr).
\end{equation}
Therefore, \eqref{eq:expan_in_la}, \eqref{eq:kn_estim} and \eqref{eq:genf}, we have
\begin{align*}
	\|u_t(\cdot,\la)-\theta\|_E&\le\sum\limits_{n\ge1}\theta C_n\dfrac{|\la|^n\|\xi\|_E^n}{n!}e^{-\gamma t}\\
		&=\theta H(|\la| \|\xi\|_E)e^{-\gamma t},\qquad |\la|\|\xi\|_E < \ln(\dfrac{\gamma}{4\beta}+1).
\end{align*}
This proves the theorem.
\end{proof}

\begin{remark}
	Note that the estimate $|\la| \|\xi\|_E < \ln\Bigl(\dfrac{\gamma}{4\beta}+1\Bigr)$ holds if and only if the initial condition satisfies 
\[
	\theta e^{-\bigl( \frac{\gamma}{4\beta}+1 \bigr)} < u_0(x) < \theta e^{ \bigl(\frac{\gamma}{4\beta}+1\bigr)}.
\]
\end{remark}
\begin{corollary}
	Let $\xi :\X\times\Omega \rightarrow \R$ be a random field. Under the assumptions of Theorem~\ref{thm:BDLP_theta_exp_stable_anal_dec}, the following estimate holds
	\begin{equation}\label{eq:sec_mom_est}
		\E \|u_t(\om,\la) - \theta\|_E^2 \le \theta^2 e^{-2\gamma t} \left(\dfrac{\gamma}{2\beta} - \E \sqrt{\frac{\gamma^2}{4\beta^2} - \bigl(e^{|\la| \|\xi(\om)\|_E}-1\bigr)\dfrac{\gamma}{\beta}}\right)^2,
	\end{equation}
	where $\sup\limits_{\om\in\Omega} |\la| \|\xi(\om)\|_E < \ln\Bigl(\dfrac{\gamma}{4\beta}+1\Bigr)$.
\end{corollary}

We apply now the general results to the specific case of the random initial data and try to estimate the rate of convergence using $L^2$-norm over a probability space.
Let us denote, for any  $f\in L^1(\X)$, its Fourier transform by 
\[
\widehat{f}(\lambda)=\intred e^{-i\la x}f(x)dx, \quad \la\in\X.
\]
Let $\tilde{p}_t(x)$ be a transition probability density for the jump process with the generator $L_J$ for $J=J_\theta$ (see \eqref{eq:jumpgen} and \eqref{ass:aplus_geq_aminus}).
Introduce the following assumption%
\begin{assum}\label{ass:J_bounded}
	J_\theta \textit{ is bounded}.
\end{assum}
Assumption \eqref{ass:J_bounded}  is a sufficient condition to have $\tilde{p}_t -\delta \in L^2(\X)\cap L^{\infty}(\X)$, $t\ge 0$.
The following theorem improves the estimate \eqref{eq:sec_mom_est}, when $J_\theta$ is non-negative.
\begin{theorem}\label{thm:random_field}
	Let \eqref{ass:aplus_geq_aminus} holds with $\k = \theta$.
	Let $\xi(x,\om)$ be a homogeneous random field with the following correlation function
	\[
		B(x-y)=\E\xi(x)\xi(y),\quad x,y\in\X.
	\]
	Suppose that $B\in L^1(\X)$ and its Fourier transform $\widehat{B}$ satisfies the following asymptotic
	\begin{equation*}
		\widehat{B}(\lambda) \sim \dfrac{a}{|\lambda|^{\alpha}},\quad \lambda\rightarrow0,
	\end{equation*}
	where $\alpha\in(0,d]$, $a>0$.
	Suppose, moreover, that the function $\widehat{J}$ is such that the following estimate
	\begin{equation*}
		\widehat{J}(\lambda)=1-b|\lambda|^{\beta}+o(|\lambda|^{\beta}),\quad \lambda\rightarrow0,
	\end{equation*}
	where $\beta\in(0,2]$, $b>0$, and let the function $x\rightarrow\sup\limits_{|\lambda|\leq x}\widehat{J}(\lambda)$ be monotonically decreasing in a neighborhood of 0.
	Then the following inequality holds
	\begin{equation*}
		\E k_{1,t}^2 \le \theta^2e^{-2\beta t}(D_1t^{\frac{\alpha-d}{\beta}}+D_2e^{-2\Delta t}),
	\end{equation*}
	where $D_1$, $D_2$, $\Delta$ are some fixed positive constants.
\end{theorem}
\begin{proof}
By assumptions of the theorem, $\tilde{p}_t*B\in L^1(\X)\cap L^{2}(\X)$. Therefore,
\begin{align*}
	\E k_{1,t}^2&=\theta^2e^{-2\beta t}\E \intred \tilde{p}_t(y)\xi(y)dy\intred \tilde{p}_t(z)\xi(z)dz\\
		&=\theta^2e^{-2\beta t}\intred \tilde{p}_t(y)\tilde{p}_t(z)\E \xi(z)\xi(y)dydz =\theta^2e^{-2\beta t}\intred \tilde{p}_t(y)\tilde{p}_t(z)B(y-z)dydz \\
		&=\theta^2e^{-2\beta t}\intred \widehat{\tilde{p}_t}(\lambda) \widehat{(\tilde{p}_t*B)}(\lambda)d\lambda =\theta^2e^{-2\beta t}\intred \left(\widehat{\tilde{p}_t}(\lambda)\right)^2 \widehat{B}(\lambda)d\lambda \nonumber \\
&=\theta^2e^{-2\beta t}\intred e^{2(\widehat{J}(\lambda)-1)t}\widehat{B}(\lambda)d\lambda,
\end{align*}
where Parseval's theorem were used.

Therefore, by assumption on $\widehat{B}$ and $\widehat{J}$ there exist $\varepsilon>0$, $\delta>0$ and $\Delta>0$ such that
\begin{align*}
	\intred e^{2(\widehat{J}(\lambda)-1)t}\widehat{B}(\lambda)d\lambda&\leq \int\limits_{B_{\delta}(0)}\dfrac{a(1+\varepsilon)}{|\lambda|^{\alpha}}e^{-2(b-\varepsilon)|\lambda|^{\beta}t}d\lambda+\int\limits_{\X\backslash B_{\delta}(0)}e^{-2\Delta t}\widehat{B}(\lambda)d\lambda  \\
		&\leq D_1t^{\frac{\alpha-d}{\beta}}+D_2e^{-2\Delta t},
\end{align*}
where $D_1$ and $D_2$ are some constants, that yields the statement.
\end{proof}

\section*{Acknowledgments}
Financial supports by the DFG through CRC 701, Research Group ``Stochastic
Dynamics: Mathematical Theory and Applications'', and by the European
Commission under the project STREVCOMS PIRSES-2013-612669 are gratefully acknowledged.

\bibliographystyle{is-abbrv}

\end{document}